\documentclass{amsart}

\newtheorem{theorem}{Theorem}[section]

\theoremstyle{definition}

\theoremstyle{remark}

\numberwithin{equation}{section}

\begin{document}

\title{On the Main Equation of Inverse Sturm-Liouville Operator with Discontinuous Coefficient}


\author{Kh. R. Mamedov}
\address{Khanlar Residoglu Mamedov,
\newline\hphantom{iii} Science and Letter Faculty,
\newline\hphantom{iii} Mathematics Department,
\newline\hphantom{iii} Mersin University, 
\newline\hphantom{iii} 333343, Mersin, Turkey}
\curraddr{}
\email{hanlar@mersin.edu.tr}
\thanks{}

\author{D. Karahan}
\address{Done Karahan,
\newline\hphantom{iii} Science and Letters Faculty, 
\newline\hphantom{iii} Mathematics Department,
\newline\hphantom{iii} Harran University, 
\newline\hphantom{iii} Sanliurfa, Turkey}
\email{dkarahan@harran.edu.tr}
\curraddr{}
\email{}
\thanks{}



\date{}

\dedicatory{}

\begin{abstract}
In this work, a boundary value problem for Sturm-Liouville operator with discontinuous coefficient is examined. The main equation is obtained which has an important role in solution of inverse problem for boundary value problem and uniqueness of its solution is proved. Uniqueness theorem for the solution of the inverse problem is given.
\end{abstract}

\maketitle

\section{Introduction}

In many practices, spectral problems are faced for differential equations
which have discontinuous coefficient and discontinuity conditions in
interval (\cite{Tik-Sam}-\cite{Sed}). These problems generally emerge in
physics, mechanics and geophysics in non-homogeneous and discontinuous
environments.

We consider a heat problem in a rod which is composed of materials having
different densities. In the initial time, let the temperature is given
arbitrary. Let be the temperature is zero in one end of the rod and the heat
is isolated at the other end of the rod. In this case the heat flow in
non-homogeneous rod is expressed with the following boundary problem:

\begin{equation*}
\rho \left( x\right) \frac{\partial u}{\partial t}=\frac{\partial ^{2}u}{%
\partial x^{2}}+q\left( x\right) u,\text{ \ \ \ \ }0<x<\pi ,\text{ }t>0,
\end{equation*}%
\begin{equation*}
\left. \frac{\partial u}{\partial x}\right\vert _{x=0}=0,\text{ \ \ \ }%
\left. u\right\vert _{x=\pi }=0,\text{ \ \ }t>0,
\end{equation*}%
where $\rho \left( x\right) $, $q\left( x\right) $ are physical parameters
and have specific properties. For instance, $\rho \left( x\right) $ defines
the density of the material and piecewise-continuous function. Applying the
method of separation of variables to this problem, we get the spectral
problem below:

\begin{equation}
-y^{\prime \prime }+q(x)y=\lambda ^{2}\rho (x)y,\ 0\leq x\leq \pi   \label{1}
\end{equation}
\begin{equation}
y^{\prime }(0)=y(\pi )=0,  \label{2}
\end{equation}
here $q\left( x\right) \in L_{2}\left( 0,\pi \right) $ is a real-valued
function, $\rho \left( x\right) $ piecewise-continuous function the
following:

\begin{equation}
\rho (x)=\left\{ 
\begin{array}{c}
1,\ \ 0\leq x\leq a, \\ 
\alpha ^{2},\ a<x\leq \pi%
\end{array}%
\right.   \label{3}
\end{equation}
$\lambda $ is spectral parameter and $a(1+\alpha )>\pi \alpha $.

When $\rho \left( x\right) \equiv 1$ or $\alpha =1$, that is, in continuous
case, the solution of inverse problem is given in \cite{Mar1}-\cite{Car}. The
spectral properties of Sturm-Liouville opeator with discontinuos coefficient in different boundary conditions are examined in \cite{Nab-Ami}-\cite{Alt-Kad-Muk}.

In this study, the main equation is obtained which has an important role in
solution of inverse problem for boundary value problem and according to
spectral data, the uniqueness of solution of inverse problem is proved. Similar problems are examined for the equation (\ref{1}) with different boundary conditions in \cite{Akh-Hus2}.

It was proved (see \cite{Akh}), that the solution $\varphi (x,\lambda )$ of
the equation (\ref{1}) with initial conditions $\varphi (0,\lambda )=1,$ $%
\varphi ^{\prime }(0,\lambda )=0$ can be represented as

\begin{equation}
\varphi (x,\lambda )=\varphi _{0}(x,\lambda )+\int_{0}^{\mu
^{+}(x)}A(x,t)\cos \lambda tdt,  \label{4}
\end{equation}%
where $A(x,t)$ belongs to the space $L_{2,\rho }\left( 0,\pi \right) $ for
each fixed $x\in \lbrack 0,\pi ]$ and is related with the coefficient $q(x)$
of the equation (\ref{1}) by the formula:%
\begin{equation}
\frac{d}{dx}A(x,\mu ^{+}(x))=\frac{1}{4\sqrt{\rho (x)}}\left( 1+\frac{1}{%
\sqrt{\rho (x)}}\right) q(x),  \label{5}
\end{equation}%
\begin{equation}
\varphi _{0}(x,\lambda )=\frac{1}{2}\left( 1+\frac{1}{\sqrt{\rho (x)}}%
\right) \cos \lambda \mu ^{+}(x)+\frac{1}{2}\left( 1-\frac{1}{\sqrt{\rho (x)}%
}\right) \cos \lambda \mu ^{-}(x)  \label{6}
\end{equation}%
is the solution of (\ref{1}) when $q(x)\equiv 0,$%
\begin{equation}
\mu ^{+}(x)=\pm x\sqrt{\rho (x)}+a\left( 1\mp \sqrt{\rho (x)}\right)
\label{7}
\end{equation}%
It is similarly shown in \cite{Akh-Hus2}, \cite{Kar-Mam}  that the roots of the equation $\varphi (\pi ,\lambda )=0$ have the following form 
\begin{equation*}
\lambda _{n}=\lambda _{n}^{0}+\frac{d_{n}}{\lambda _{n}^{0}}+\frac{k_{n}}{n},%
\text{ }\lambda _{n}\geq 0,
\end{equation*}%
where $\left\{ \lambda _{n}^{0}\right\} ^{2}$ are the eigenvalues of problem
(\ref{1}), (\ref{2}) when $q(x)\equiv 0$, $d_{n}$ is a bounded sequence, $%
k_{n}\in l_{2}$ and norming constants:
\begin{equation*}
\alpha _{n}=\int_{0}^{\pi }\rho (x)\varphi ^{2}(x,\lambda _{n})dx.
\end{equation*}

\section{Main Equation}

\begin{theorem}
For each fixed $x\in \lbrack 0,\pi ]$ the kernel $A(x,t)$ from the
representation (\ref{4}) satisfies the following linear functional integral
equation 
\begin{equation}
\frac{2}{1+\sqrt{\rho (t)}}A\left( x,\mu ^{+}(t)\right) +\frac{1-\sqrt{\rho
(2a-t)}}{1+\sqrt{\rho (2a-t)}}A\left( x,2a-t\right) +  \notag
\end{equation}%
\begin{equation}
+F(x,t)+\int_{0}^{\mu ^{+}(x)}A(x,\xi )F_{0}(\xi ,t)d\xi =0,\text{ \ \ }0<t<x
\label{8}
\end{equation}%
where%
\begin{equation}
F_{0}(x,t)=\sum_{n=1}^{\infty }\left( \frac{\varphi _{0}(t,\lambda _{n})\cos
\lambda _{n}x}{\alpha _{n}}-\frac{\varphi _{0}(t,\lambda _{n}^{0})\cos
\lambda _{n}^{0}x}{\alpha _{n}^{0}}\right)  \label{9}
\end{equation}%
\begin{equation}
F(x,t)=\frac{1}{2}\left( 1+\frac{1}{\sqrt{\rho (x)}}\right) F_{0}(\mu
^{+}(x),t)+\frac{1}{2}\left( 1-\frac{1}{\sqrt{\rho (x)}}\right) F_{0}(\mu
^{-}(x),t)  \label{10}
\end{equation}%
$\left\{ \lambda _{n}^{0}\right\} ^{2}$ are eigenvalues and $\alpha _{n}^{0}$
are norming constants of the boundary value problem (\ref{1}), (\ref{2})
when $q(x)\equiv 0.$
\end{theorem}

\begin{proof}
From (\ref{4}) we have%
\begin{equation}
\varphi _{0}(x,\lambda )=\varphi (x,\lambda )-\int_{0}^{\mu
^{+}(x)}A(x,t)\cos \lambda tdt.  \label{11}
\end{equation}%
It follows from (\ref{4}) and (\ref{11}) that%
\begin{equation*}
\sum_{n=1}^{N}\frac{\varphi (x,\lambda _{n})\varphi _{0}(t,\lambda _{n})}{%
\alpha _{n}}=\sum_{n=1}^{N}\left( \frac{\varphi _{0}(x,\lambda _{n})\varphi
_{0}(t,\lambda _{n})}{\alpha _{n}}+\right.
\end{equation*}%
\begin{equation*}
+\left. \frac{\varphi _{0}(t,\lambda _{n})}{\alpha _{n}}\int_{0}^{\mu
^{+}(x)}A(x,\xi )\cos \lambda _{n}\xi d\xi \right) =
\end{equation*}%
\begin{equation*}
=\sum_{n=1}^{N}\left( \frac{\varphi _{0}(x,\lambda _{n})\varphi
_{0}(t,\lambda _{n})}{\alpha _{n}}-\frac{\varphi _{0}(x,\lambda
_{n}^{0})\varphi _{0}(t,\lambda _{n}^{0})}{\alpha _{n}^{0}}\right) +
\end{equation*}%
\begin{equation*}
+\sum_{n=1}^{N}\frac{\varphi _{0}(x,\lambda _{n}^{0})\varphi _{0}(t,\lambda
_{n}^{0})}{\alpha _{n}^{0}}+
\end{equation*}%
\begin{equation*}
+\int_{0}^{\mu ^{+}(x)}A(x,\xi )\sum_{n=1}^{N}\left( \frac{\varphi
_{0}(t,\lambda _{n})\cos \lambda _{n}\xi }{\alpha _{n}}-\frac{\varphi
_{0}(t,\lambda _{n}^{0})\cos \lambda _{n}^{0}\xi }{\alpha _{n}^{0}}\right)
d\xi +
\end{equation*}%
\begin{equation*}
+\int_{0}^{\mu ^{+}(x)}A(x,\xi )\sum_{n=1}^{N}\frac{\varphi _{0}(t,\lambda
_{n}^{0})\cos \lambda _{n}^{0}\xi }{\alpha _{n}^{0}}d\xi
\end{equation*}%
\begin{equation*}
\sum_{n=1}^{N}\frac{\varphi (x,\lambda _{n})\varphi _{0}(t,\lambda _{n})}{%
\alpha _{n}}=\sum_{n=1}^{N}\frac{\varphi (x,\lambda _{n})\varphi (t,\lambda
_{n})}{\alpha _{n}}-
\end{equation*}%
\begin{equation*}
-\int_{0}^{\mu ^{+}(t)}A(t,\xi )\sum_{n=1}^{N}\frac{\varphi (x,\lambda
_{n})\cos \lambda _{n}\xi }{\alpha _{n}}d\xi .
\end{equation*}%
Using the last two equalities, we obtain%
\begin{equation*}
\sum_{n=1}^{N}\left( \frac{\varphi (x,\lambda _{n})\varphi (t,\lambda _{n})}{%
\alpha _{n}}-\frac{\varphi _{0}(x,\lambda _{n}^{0})\varphi _{0}(t,\lambda
_{n}^{0})}{\alpha _{n}^{0}}\right) =
\end{equation*}%
\begin{equation*}
=\sum_{n=1}^{N}\left( \frac{\varphi _{0}(x,\lambda _{n})\varphi
_{0}(t,\lambda _{n})}{\alpha _{n}}-\frac{\varphi _{0}(x,\lambda
_{n}^{0})\varphi _{0}(t,\lambda _{n}^{0})}{\alpha _{n}^{0}}\right) +
\end{equation*}%
\begin{equation*}
+\int_{0}^{\mu ^{+}(x)}A(x,\xi )\sum_{n=1}^{N}\frac{\varphi _{0}(t,\lambda
_{n}^{0})\cos \lambda _{n}^{0}\xi }{\alpha _{n}^{0}}d\xi +
\end{equation*}%
\begin{equation*}
+\int_{0}^{\mu ^{+}(x)}A(x,\xi )\sum_{n=1}^{N}\left( \frac{\varphi
_{0}(t,\lambda _{n})\cos \lambda _{n}\xi }{\alpha _{n}}-\frac{\varphi
_{0}(t,\lambda _{n}^{0})\cos \lambda _{n}^{0}\xi }{\alpha _{n}^{0}}\right)
d\xi +
\end{equation*}%
\begin{equation*}
+\int_{0}^{\mu ^{+}(t)}A(t,\xi )\sum_{n=1}^{N}\frac{\varphi (x,\lambda
_{n})\cos \lambda _{n}\xi }{\alpha _{n}}d\xi ,
\end{equation*}%
or%
\begin{equation}
\Phi _{N}(x,t)=I_{N1}(x,t)+I_{N2}(x,t)+I_{N3}(x,t)+I_{N4}(x,t),  \label{12}
\end{equation}%
where%
\begin{equation*}
\Phi _{N}(x,t):=\sum_{n=1}^{N}\left( \frac{\varphi (x,\lambda _{n})\varphi
(t,\lambda _{n})}{\alpha _{n}}-\frac{\varphi _{0}(x,\lambda _{n}^{0})\varphi
_{0}(t,\lambda _{n}^{0})}{\alpha _{n}^{0}}\right) ,
\end{equation*}%
\begin{equation*}
I_{N1}(x,t):=\sum_{n=1}^{N}\left( \frac{\varphi _{0}(x,\lambda _{n})\varphi
_{0}(t,\lambda _{n})}{\alpha _{n}}-\frac{\varphi _{0}(x,\lambda
_{n}^{0})\varphi _{0}(t,\lambda _{n}^{0})}{\alpha _{n}^{0}}\right) ,
\end{equation*}%
\begin{equation*}
I_{N2}(x,t):=\int_{0}^{\mu ^{+}(x)}A(x,\xi )\sum_{n=1}^{N}\frac{\varphi
_{0}(t,\lambda _{n}^{0})\cos \lambda _{n}^{0}\xi }{\alpha _{n}^{0}}d\xi ,
\end{equation*}%
\begin{equation*}
I_{N3}(x,t):=\int_{0}^{\mu ^{+}(x)}A(x,\xi )\sum_{n=1}^{N}\left( \frac{%
\varphi _{0}(t,\lambda _{n})\cos \lambda _{n}\xi }{\alpha _{n}}-\frac{%
\varphi _{0}(t,\lambda _{n}^{0})\cos \lambda _{n}^{0}\xi }{\alpha _{n}^{0}}%
\right) d\xi ,
\end{equation*}%
\begin{equation*}
I_{N4}(x,t):=\int_{0}^{\mu ^{+}(t)}A(t,\xi )\sum_{n=1}^{N}\frac{\varphi
(x,\lambda _{n})\cos \lambda _{n}\xi }{\alpha _{n}}d\xi .
\end{equation*}%
It is easily found by using (\ref{9}) and (\ref{10})%
\begin{equation*}
F(x,t)=\sum_{n=1}^{\infty }\left( \frac{\varphi _{0}(x,\lambda _{n})\varphi
_{0}(t,\lambda _{n})}{\alpha _{n}}-\frac{\varphi _{0}(x,\lambda
_{n}^{0})\varphi _{0}(t,\lambda _{n}^{0})}{\alpha _{n}^{0}}\right) .
\end{equation*}%
Let $f(x)$ be an absolutely continuous function, $f^{\prime }(0)=f(\pi )=0.$
Then using expansion formula (see \cite{Kar-Mam}),%
\begin{equation*}
\sum_{n=1}^{\infty }\int_{0}^{\pi }f(t)\rho (t)\frac{\varphi (x,\lambda
_{n})\varphi (t,\lambda _{n})}{\alpha _{n}}dt=f(x),
\end{equation*}%
\begin{equation}
\sum_{n=1}^{\infty }\int_{0}^{\pi }f(t)\rho (t)\frac{\varphi _{0}(x,\lambda
_{n}^{0})\varphi _{0}(t,\lambda _{n}^{0})}{\alpha _{n}^{0}}dt=f(x).
\label{13}
\end{equation}%
Using (\ref{13}) we have:%
\begin{equation*}
\lim_{N\rightarrow \infty }\max_{0\leq x\leq \pi }\left\vert \int_{0}^{\pi
}f(t)\rho (t)\Phi _{N}(x,t)dt\right\vert =
\end{equation*}%
\begin{equation*}
=\lim_{N\rightarrow \infty }\max_{0\leq x\leq \pi }\left\vert \int_{0}^{\pi
}f(t)\rho (t)\sum_{n=1}^{N}\left( \frac{\varphi (x,\lambda _{n})\varphi
(t,\lambda _{n})}{\alpha _{n}}-\frac{\varphi _{0}(x,\lambda _{n}^{0})\varphi
_{0}(t,\lambda _{n}^{0})}{\alpha _{n}^{0}}\right) dt\right\vert \leq
\end{equation*}%
\begin{equation*}
\leq \lim_{N\rightarrow \infty }\left\{ \max_{0\leq x\leq \pi }\left\vert
\int_{0}^{\pi }f(t)\rho (t)\sum_{n=1}^{N}\frac{\varphi (x,\lambda
_{n})\varphi (t,\lambda _{n})}{\alpha _{n}}dt-f(x)\right\vert +\right.
\end{equation*}%
\begin{equation}
+\left. \max_{0\leq x\leq \pi }\left\vert \int_{0}^{\pi }f(t)\rho
(t)\sum_{n=1}^{N}\frac{\varphi _{0}(x,\lambda _{n}^{0})\varphi
_{0}(t,\lambda _{n}^{0})}{\alpha _{n}^{0}}dt-f(x)\right\vert \right\} =0.
\label{14}
\end{equation}%
We obtain uniformly on $x\in \lbrack 0,\pi ]$%
\begin{equation*}
\lim_{N\rightarrow \infty }\int_{0}^{\pi }f(t)\rho (t)I_{N1}(x,t)dt=
\end{equation*}%
\begin{equation*}
=\lim_{N\rightarrow \infty }\int_{0}^{\pi }f(t)\rho (t)\sum_{n=1}^{N}\left( 
\frac{\varphi _{0}(x,\lambda _{n})\varphi _{0}(t,\lambda _{n})}{\alpha _{n}}-%
\frac{\varphi _{0}(x,\lambda _{n}^{0})\varphi _{0}(t,\lambda _{n}^{0})}{%
\alpha _{n}^{0}}\right) dt=
\end{equation*}%
\begin{equation}
=\int_{0}^{\pi }f(t)\rho (t)F(x,t)dt.  \label{15}
\end{equation}%
It follows from (\ref{6}) that 
\begin{equation}
\cos \lambda \xi =\left\{ 
\begin{array}{c}
\varphi _{0}(\xi ,\lambda )\quad,\quad \xi <a, \\ 
\frac{2\alpha }{1+\alpha }\varphi _{0}(\frac{\xi }{\alpha }+a-\frac{a}{
\alpha },\lambda )+\frac{1-\alpha }{1+\alpha }\varphi _{0}(2a-\xi ,\lambda )\quad,\quad \xi >a.%
\end{array}
\right.  \label{16}
\end{equation}
Taking into account (\ref{16}) and (\ref{13}), we get 
\begin{equation*}
\lim_{N\rightarrow \infty }\int_{0}^{\pi }f(t)\rho (t)I_{N2}(x,t)dt=
\end{equation*}
\begin{equation*}
=\lim_{N\rightarrow \infty }\int_{0}^{\pi }f(t)\rho (t)\int_{0}^{\mu
^{+}(x)}A(x,\xi )\sum_{n=1}^{N}\frac{\varphi _{0}(t,\lambda _{n}^{0})\cos
\lambda _{n}^{0}\xi }{\alpha _{n}^{0}}d\xi dt=
\end{equation*}%
\begin{equation*}
=\lim_{N\rightarrow \infty }\int_{0}^{\pi }f(t)\rho (t)\int_{0}^{a}A(x,\xi
)\sum_{n=1}^{N}\frac{\varphi _{0}(t,\lambda _{n}^{0})\varphi _{0}(\xi
,\lambda )}{\alpha _{n}^{0}}d\xi dt+
\end{equation*}%
\begin{equation*}
+\frac{2\alpha }{1+\alpha }\lim_{N\rightarrow \infty }\int_{0}^{\pi
}f(t)\rho (t)\int_{a}^{\alpha x-\alpha a+a}A(x,\xi )\times
\end{equation*}%
\begin{equation*}
\times \sum_{n=1}^{N}\frac{\varphi _{0}(t,\lambda _{n}^{0})\varphi _{0}(%
\frac{\xi }{\alpha }+a-\frac{a}{\alpha },\lambda _{n}^{0})}{\alpha _{n}^{0}}%
d\xi dt+
\end{equation*}%
\begin{equation*}
+\frac{1-\alpha }{1+\alpha }\lim_{N\rightarrow \infty }\int_{0}^{\pi
}f(t)\rho (t)\int_{a}^{\alpha x-\alpha a+a}A(x,\xi )\times
\end{equation*}%
\begin{equation*}
\times \sum_{n=1}^{N}\frac{\varphi _{0}(t,\lambda _{n}^{0})\varphi
_{0}(2a-\xi ,\lambda _{n}^{0})}{\alpha _{n}^{0}}d\xi dt=
\end{equation*}%
\begin{equation*}
=\int_{0}^{a}A(x,\xi )\int_{0}^{\pi }f(t)\rho (t)\sum_{n=1}^{\infty }\frac{%
\varphi _{0}(t,\lambda _{n}^{0})\varphi _{0}(\xi ,\lambda _{n}^{0})}{\alpha
_{n}^{0}}dtd\xi +
\end{equation*}%
\begin{equation*}
+\frac{2\alpha }{1+\alpha }\int_{a}^{\alpha x-\alpha a+a}A(x,\xi
)\int_{0}^{\pi }f(t)\rho (t)\times
\end{equation*}%
\begin{equation*}
\times \sum_{n=1}^{\infty }\frac{\varphi _{0}(t,\lambda _{n}^{0})\varphi
_{0}(\frac{\xi }{\alpha }+a-\frac{a}{\alpha },\lambda _{n}^{0})}{\alpha
_{n}^{0}}dtd\xi +
\end{equation*}%
\begin{equation*}
+\frac{1-\alpha }{1+\alpha }\int_{a}^{\alpha x-\alpha a+a}A(x,\xi
)\int_{0}^{\pi }f(t)\rho (t)\times
\end{equation*}%
\begin{equation*}
\times \sum_{n=1}^{\infty }\frac{\varphi _{0}(t,\lambda _{n}^{0})\varphi
_{0}(2a-\xi ,\lambda _{n}^{0})}{\alpha _{n}^{0}}dtd\xi =
\end{equation*}%
\begin{equation*}
=\int_{0}^{a}A(x,\xi )f(\xi )d\xi +\frac{2\alpha }{1+\alpha }%
\int_{a}^{\alpha x-\alpha a+a}A(x,\xi )f\left( \frac{\xi }{\alpha }+a-\frac{a%
}{\alpha }\right) d\xi +
\end{equation*}%
\begin{equation*}
+\frac{1-\alpha }{1+\alpha }\int_{a}^{\alpha x-\alpha a+a}A(x,\xi )f\left(
2a-\xi \right) d\xi .
\end{equation*}%
Substituting $\frac{\xi }{\alpha }+a-\frac{a}{\alpha }\rightarrow \xi
^{\prime }$ and $2a-\xi \rightarrow \xi ^{\prime \prime }$ we obtain 
\begin{equation*}
\lim_{N\rightarrow \infty }\int_{0}^{\pi }f(t)\rho
(t)I_{N2}(x,t)dt=\int_{0}^{a}A(x,\xi )f\left( \xi \right) d\xi +
\end{equation*}%
\begin{equation*}
+\frac{2\alpha ^{2}}{1+\alpha }\int_{a}^{x}A\left( x,\alpha \xi ^{\prime
}-\alpha a+a\right) f\left( \xi ^{\prime }\right) d\xi ^{\prime }+
\end{equation*}%
\begin{equation*}
+\frac{1-\alpha }{1+\alpha }\int_{-\alpha x+\alpha a+a}^{a}A(x,2a-\xi
^{\prime \prime })f\left( \xi ^{\prime \prime }\right) d\xi ^{\prime \prime
}.
\end{equation*}%
Since $A(x,2a-\xi ^{\prime \prime })\equiv 0$ when $2a-\xi >\alpha x-\alpha
a+a$, we have 
\begin{equation*}
\lim_{N\rightarrow \infty }\int_{0}^{\pi }f(t)\rho
(t)I_{N2}(x,t)dt=\int_{0}^{a}A(x,t)f\left( t\right) dt+
\end{equation*}%
\begin{equation*}
+\frac{2\alpha ^{2}}{1+\alpha }\int_{a}^{x}A\left( x,\alpha t-\alpha
a+a\right) f\left( t\right) dt+
\end{equation*}%
\begin{equation*}
+\frac{1-\alpha }{1+\alpha }\int_{0}^{a}A(x,2a-t)f\left( t\right) dt.
\end{equation*}%
Thus, uniformly on $x\in \lbrack 0,\pi ]:$%
\begin{equation*}
\lim_{N\rightarrow \infty }\int_{0}^{\pi }f(t)\rho
(t)I_{N2}(x,t)dt=\int_{0}^{x}\frac{2\rho \left( t\right) }{1+\sqrt{\rho
\left( t\right) }}A(x,\mu ^{+}(t))f\left( t\right) dt+
\end{equation*}%
\begin{equation}
+\int_{0}^{x}\frac{1-\sqrt{\rho \left( 2a-t\right) }}{1+\sqrt{\rho \left(
2a-t\right) }}A(x,2a-t)f\left( t\right) dt.  \label{17}
\end{equation}%
Using (\ref{9}), uniformly on $x\in \lbrack 0,\pi ]$%
\begin{equation*}
\lim_{N\rightarrow \infty }\int_{0}^{\pi }f(t)\rho (t)I_{N3}(x,t)dt=
\end{equation*}%
\begin{equation*}
=\lim_{N\rightarrow \infty }\int_{0}^{\pi }f(t)\rho (t)\int_{0}^{\mu
^{+}(x)}A(x,\xi )\sum_{n=1}^{N}\left( \frac{\varphi _{0}(t,\lambda _{n})\cos
\lambda _{n}\xi }{\alpha _{n}}-\right.
\end{equation*}%
\begin{equation*}
\left. -\frac{\varphi _{0}(t,\lambda _{n}^{0})\cos \lambda _{n}^{0}\xi }{%
\alpha _{n}^{0}}\right) d\xi dt=
\end{equation*}%
\begin{equation*}
=\int_{0}^{\pi }f(t)\rho (t)\int_{0}^{\mu ^{+}(x)}A(x,\xi
)\sum_{n=1}^{\infty }\left( \frac{\varphi _{0}(t,\lambda _{n})\cos \lambda
_{n}\xi }{\alpha _{n}}-\right.
\end{equation*}%
\begin{equation*}
\left. -\frac{\varphi _{0}(t,\lambda _{n}^{0})\cos \lambda _{n}^{0}\xi }{%
\alpha _{n}^{0}}\right) d\xi dt=
\end{equation*}%
\begin{equation}
=\int_{0}^{\pi }f(t)\rho (t)\int_{0}^{\mu ^{+}(x)}A(x,\xi )F_{0}(\xi ,t)d\xi
dt.  \label{18}
\end{equation}%
Using the residue theorem and the formula $\frac{\varphi \left( x,\lambda
_{n}\right) }{2\lambda _{n}\alpha _{n}}=\frac{\psi \left( x,\lambda
_{n}\right) }{\overset{.}{\Delta }(\lambda _{n})}$ (see \cite{Kar-Mam}),
where $\psi \left( x,\lambda \right) $ is the solution of (\ref{1}) with
initial condition $\psi \left( \pi ,\lambda \right) =0,$ $\psi ^{\prime
}\left( \pi ,\lambda \right) =1$ and $\Delta \left( \lambda \right) =\varphi
\left( \pi ,\lambda \right) $ is the characteristic function of (\ref{1})-(%
\ref{3}), $\overset{.}{\Delta }(\lambda )=\frac{d}{d\lambda }\Delta \left(
\lambda \right) ,$ we calculate%
\begin{equation*}
\lim_{N\rightarrow \infty }\int_{0}^{\pi }f(t)\rho (t)I_{N4}(x,t)dt=
\end{equation*}%
\begin{equation*}
=\lim_{N\rightarrow \infty }\int_{0}^{\pi }f(t)\rho (t)\int_{0}^{\mu
^{+}(t)}A(t,\xi )\sum_{n=1}^{N}\frac{\varphi (x,\lambda _{n})\cos \lambda
_{n}\xi }{\alpha _{n}}d\xi dt=
\end{equation*}%
\begin{equation*}
=2\lim_{N\rightarrow \infty }\int_{0}^{\pi }f(t)\rho (t)\sum_{\left\vert
\lambda _{n}\right\vert \leq N}\lambda _{n}\frac{\psi \left( x,\lambda
_{n}\right) }{\overset{.}{\Delta }(\lambda _{n})}\int_{0}^{\mu
^{+}(t)}A(t,\xi )\cos \lambda _{n}\xi d\xi dt=
\end{equation*}%
\begin{equation*}
=2\lim_{N\rightarrow \infty }\int_{0}^{\pi }f(t)\rho (t)\sum_{\left\vert
\lambda _{n}\right\vert \leq N}\underset{\lambda = \lambda _{n}}{Res}
\left[ \lambda \frac{\psi \left( x,\lambda \right) }{\Delta (\lambda )}
\int_{0}^{\mu ^{+}(t)}A(t,\xi )\cos \lambda \xi d\xi \right] dt=
\end{equation*}%
\begin{equation*}
=2\lim_{N\rightarrow \infty }\int_{0}^{\pi }f(t)\rho (t)\frac{1}{2\pi i}%
\oint_{\Gamma _{N}}\lambda \frac{\psi \left( x,\lambda \right) }{\Delta
(\lambda )}\int_{0}^{\mu ^{+}(t)}A(t,\xi )\cos \lambda \xi d\xi d\lambda dt=
\end{equation*}%
\begin{equation*}
=\lim_{N\rightarrow \infty }\int_{0}^{\pi }f(t)\rho (t)\frac{1}{\pi i}%
\oint_{\Gamma _{N}}\lambda \frac{\psi \left( x,\lambda \right) }{\Delta
(\lambda )}e^{\left\vert {Im}\lambda \right\vert \mu
^{+}(t)}e^{-\left\vert {Im}\lambda \right\vert \mu ^{+}(t)}\times
\end{equation*}%
\begin{equation*}
\times \int_{0}^{\mu ^{+}(t)}A(t,\xi )\cos \lambda \xi d\xi d\lambda dt=
\end{equation*}%
\begin{equation*}
=\int_{0}^{\pi }f(t)\rho (t)\lim_{N\rightarrow \infty }\left( \frac{1}{\pi i}%
\oint_{\Gamma _{N}}\lambda \frac{\psi \left( x,\lambda \right) }{\Delta
(\lambda )}e^{\left\vert {Im}\lambda \right\vert \mu
^{+}(t)}e^{-\left\vert {Im}\lambda \right\vert \mu ^{+}(t)}\times
\right.
\end{equation*}%
\begin{equation}
\left. \times \int_{0}^{\mu ^{+}(t)}A(t,\xi )\cos \lambda \xi d\xi d\lambda
\right) dt  \label{19}
\end{equation}%
where $\Gamma _{N}=\left\{ \lambda :\left\vert \lambda \right\vert
=N\right\} .$ Since (see \cite{Kar-Mam})%
\begin{eqnarray*}
\psi (x,\lambda ) &=&O\left( \frac{e^{\left\vert Im\lambda \right\vert (\mu
^{+}(\pi )-\mu ^{+}(x))}}{\left\vert \lambda \right\vert }\right) ,\quad
\left\vert \lambda \right\vert \rightarrow \infty , \\
\left\vert \Delta (\lambda )\right\vert &\geq &C_{\delta }e^{\left\vert
Im\lambda \right\vert \mu ^{+}(\pi )},\quad \lambda \in G_{\delta },
\end{eqnarray*}%
($G_{\delta }=\left\{ \lambda :\left\vert \lambda -\lambda _{n}\right\vert
\geq \delta \right\} ,$ $\delta $ is a sufficiently small positive
number) and according to Lemma 1.3.1 from \cite{Mar1}%
\begin{equation*}
\lim_{\left\vert \lambda \right\vert \rightarrow \infty }\max_{0\leq t\leq
\pi }e^{-\left\vert {Im}\lambda \right\vert \mu ^{+}(t)}\left\vert
\int_{0}^{\mu ^{+}(t)}A(t,\xi )\cos \lambda \xi d\xi d\lambda \right\vert =0
\end{equation*}%
from the equality (\ref{19}) we get 
\begin{equation}
\lim_{N\rightarrow \infty }\int_{0}^{\pi }f(t)\rho (t)I_{N4}(x,t)dt=0.
\label{20}
\end{equation}%
Multiplying both sides of (\ref{12}) by $\rho (x)f(x),$ integrating from $0$
to $\pi $, tending to limit when $N\rightarrow \infty $ and using (\ref{14}%
), (\ref{15}), (\ref{17}), (\ref{18}) and (\ref{20}) we have 
\begin{equation*}
\int_{0}^{x}\frac{2\rho \left( t\right) }{1+\sqrt{\rho \left( t\right) }}%
A(x,\mu ^{+}(t))f\left( t\right) dt+\int_{0}^{x}\frac{1-\sqrt{\rho \left(
2a-t\right) }}{1+\sqrt{\rho \left( 2a-t\right) }}A(x,2a-t)f\left( t\right)
dt+
\end{equation*}%
\begin{equation*}
+\int_{0}^{\pi }f(t)\rho (t)F(x,t)dt+\int_{0}^{\pi }f(t)\rho
(t)\int_{0}^{\mu ^{+}(x)}A(x,\xi )F_{0}(\xi ,t)d\xi dt=0.
\end{equation*}%
Since $f(x)$ can be chosen arbitrarily, we obtain 
\begin{equation*}
\frac{2\rho \left( t\right) }{1+\sqrt{\rho \left( t\right) }}A(x,\mu
^{+}(t))+\frac{1-\sqrt{\rho \left( 2a-t\right) }}{1+\sqrt{\rho \left(
2a-t\right) }}A(x,2a-t)+F(x,t)+
\end{equation*}%
\begin{equation*}
+\int_{0}^{\mu ^{+}(x)}A(x,\xi )F_{0}(\xi ,t)d\xi =0.
\end{equation*}
\end{proof}

\section{Theorem for the Solution of the Inverse Problem}

\begin{theorem}
For each fixed $x\in \lbrack 0,\pi ]$ main equation (\ref{8}) has a unique
solution $A(x,.)\in L_{2,\rho }\left( 0,\mu ^{+}(x)\right) $.
\end{theorem}

\begin{proof}
We show that for each fixed $x>a$ the equation (\ref{8}) is equivalent to
the equation of the form $\left( I+B\right) f=g$ where $B$ is a completely
continuous operator, $I$ is an identity operator in the space $L_{2,\rho
}(0,\pi )$. (When $x\leq a$ this fact is obvious.)

When $x>a$ rewrite (\ref{8}) as 
\begin{equation*}
L_{x}A(x,.)+K_{x}A(x,.)=-F(x,.),
\end{equation*}%
where%
\begin{equation}
\left( L_{x}f\right) (t)=\left\{ 
\begin{array}{c}
f(t)+\frac{1-\alpha }{1+\alpha }f(2a-t)\quad,\quad t\leq a<x, \\ 
\frac{2}{1+\alpha }f(\alpha t-\alpha a+a)\quad,\quad  a<t<x.
\end{array}
\right.  \label{21}
\end{equation}%
\begin{equation*}
\left( K_{x}f\right) (t)=\int_{0}^{\alpha x-\alpha a+a}f(\xi )F_{0}\left(
\xi ,t\right) d\xi ,\quad 0<t<x.
\end{equation*}%
\ It is sufficient to prove that $L_{x}$ is invertible, i.e. has a bounded
inverse in $L_{2,\rho }\left( 0,\pi \right) $.

Consider the equation $\left( L_{x}f\right) (t)=\phi (t),$ $\phi (t)\in
L_{2,\rho }\left( 0,\pi \right) ,$ i.e. 
\begin{equation*}
\left\{ 
\begin{array}{c}
f(t)+\frac{1-\alpha }{1+\alpha }f(2a-t)=\phi (t) \quad,\quad t\leq a<x, \\ 
\frac{2}{1+\alpha }f(\alpha t-\alpha a+a)=\phi (t)\quad,\quad a<t<x.
\end{array}%
\right.
\end{equation*}%
From here it is easily to obtain 
\begin{equation*}
f(t)=\left( L_{x}^{-1}\phi \right) (t)=\left\{ 
\begin{array}{c}
\phi (t)-\frac{1-\alpha }{2}\phi \left( \frac{-t+\alpha a+a}{\alpha }\right) \quad,\quad t<a \\ 
\frac{1+\alpha }{2}\phi \left( \frac{t+\alpha a-a}{\alpha }\right) \quad,\quad t>a.%
\end{array}%
\right.
\end{equation*}%
We show that 
\begin{equation*}
\left\Vert f\right\Vert _{L_{2}}=\left\Vert L_{x}^{-1}\phi \right\Vert \leq
C\left\Vert \phi \right\Vert _{L_{2}}.
\end{equation*}%
In fact, 
\begin{equation*}
\int_{0}^{\pi }\left\vert f(t)\right\vert ^{2}dt=\int_{0}^{a}\left\vert \phi
(t)-\frac{1-\alpha }{2}\phi \left( \frac{-t+\alpha a+a}{\alpha }\right)
\right\vert ^{2}dt+
\end{equation*}%
\begin{equation*}
+\int_{a}^{\pi }\left\vert \frac{1+\alpha }{2}\phi \left( \frac{t+\alpha a-a%
}{\alpha }\right) \right\vert ^{2}dt\leq 2\int_{0}^{a}\left\vert \phi
(t)\right\vert ^{2}dt+
\end{equation*}%
\begin{equation*}
+2\left( \frac{1-\alpha }{2}\right) ^{2}\int_{0}^{a}\left\vert \phi \left( 
\frac{-t+\alpha a+a}{\alpha }\right) \right\vert ^{2}dt+
\end{equation*}%
\begin{equation*}
+\left( \frac{1+\alpha }{2}\right) ^{2}\int_{a}^{\pi }\left\vert \phi \left( 
\frac{t+\alpha a-a}{\alpha }\right) \right\vert ^{2}dt\leq
\end{equation*}%
\begin{equation*}
\leq 2\int_{0}^{\pi }\left\vert \phi (t)\right\vert ^{2}dt+\frac{\alpha
\left( 1-\alpha \right) ^{2}}{2}\int_{a}^{\frac{\alpha a+a}{\alpha }%
}\left\vert \phi (t)\right\vert ^{2}dt+
\end{equation*}%
\begin{equation*}
+\alpha \left( \frac{1+\alpha }{2}\right) ^{2}\int_{a}^{\frac{\pi +\alpha a-a%
}{\alpha }}\left\vert \phi (t)\right\vert ^{2}dt.
\end{equation*}%
We put $\phi (t)=0$, when $t>\pi $. Then 
\begin{equation*}
\int_{0}^{\pi }\left\vert f(t)\right\vert ^{2}dt\leq C\int_{0}^{\pi
}\left\vert \phi (t)\right\vert ^{2}dt=C\left\Vert \phi (t)\right\Vert
_{L_{2,\rho }\left( 0,\pi \right) }.
\end{equation*}

So the operator $L_{x}$ is invertible in $L_{2,\rho }\left( 0,\pi \right) $.
Then according to Theorem 3 from \cite{Lus} (see p. 275) it is sufficient to
prove that the equation 
\begin{equation*}
\frac{2}{1+\sqrt{\rho \left( t\right) }}A(\mu ^{+}(t))+\frac{1-\sqrt{\rho
\left( 2a-t\right) }}{1+\sqrt{\rho \left( 2a-t\right) }}A(2a-t)+
\end{equation*}%
\begin{equation}
+\int_{0}^{\mu ^{+}(x)}A(\xi )F_{0}(\xi ,t)d\xi =0  \label{22}
\end{equation}%
has only trivial solution $A(t)=0.$

Let $A(t)$ be a non-trivial solution of (\ref{22}). Then

\begin{equation*}
\int_{0}^{x}\rho (t)\left( \frac{2}{1+\sqrt{\rho \left( t\right) }}A(\mu
^{+}(t))+\frac{1-\sqrt{\rho \left( 2a-t\right) }}{1+\sqrt{\rho \left(
2a-t\right) }}A(2a-t)\right) ^{2}dt+
\end{equation*}%
\begin{equation*}
+\int_{0}^{x}\rho (t)\left( \frac{2}{1+\sqrt{\rho \left( t\right) }}A(\mu
^{+}(t))+\frac{1-\sqrt{\rho \left( 2a-t\right) }}{1+\sqrt{\rho \left(
2a-t\right) }}A(2a-t)\right) \times
\end{equation*}%
\begin{equation*}
\times \int_{0}^{\mu ^{+}(x)}A(\xi )F_{0}(\xi ,t)d\xi dt=0.
\end{equation*}%
From (\ref{9}) we have

\begin{equation*}
\int_{0}^{x}\rho (t)\left( \frac{2}{1+\sqrt{\rho \left( t\right) }}A(\mu
^{+}(t))+\frac{1-\sqrt{\rho \left( 2a-t\right) }}{1+\sqrt{\rho \left(
2a-t\right) }}A(2a-t)\right) ^{2}dt+
\end{equation*}%
\begin{equation*}
+\int_{0}^{x}\frac{2\rho (t)}{1+\sqrt{\rho \left( t\right) }}A(\mu
^{+}(t))\int_{0}^{\mu ^{+}(x)}A(\xi )\times
\end{equation*}%
\begin{equation*}
\times \sum_{n=1}^{\infty }\left( \frac{\varphi _{0}(t,\lambda _{n})\cos
\lambda _{n}\xi }{\alpha _{n}}-\frac{\varphi _{0}(t,\lambda _{n}^{0})\cos
\lambda _{n}^{0}\xi }{\alpha _{n}^{0}}\right) d\xi dt+
\end{equation*}%
\begin{equation*}
+\int_{0}^{x}\frac{1-\sqrt{\rho \left( 2a-t\right) }}{1+\sqrt{\rho \left(
2a-t\right) }}A(2a-t)\int_{0}^{\mu ^{+}(x)}A(\xi )\times
\end{equation*}%
\begin{equation*}
\times \sum_{n=1}^{\infty }\left( \frac{\varphi _{0}(t,\lambda _{n})\cos
\lambda _{n}\xi }{\alpha _{n}}-\frac{\varphi _{0}(t,\lambda _{n}^{0})\cos
\lambda _{n}^{0}\xi }{\alpha _{n}^{0}}\right) d\xi dt=0.
\end{equation*}%
Using (\ref{7}) and (\ref{16}) we obtain

\begin{equation*}
\int_{0}^{x}\rho (t)\left( \frac{2}{1+\sqrt{\rho \left( t\right) }}A(\mu
^{+}(t))+\frac{1-\sqrt{\rho \left( 2a-t\right) }}{1+\sqrt{\rho \left(
2a-t\right) }}A(2a-t)\right) ^{2}dt+
\end{equation*}%
\begin{equation*}
+\int_{0}^{x}\frac{2\rho (t)}{1+\sqrt{\rho \left( t\right) }}A(\mu
^{+}(t))\int_{0}^{a}A(\xi )\sum_{n=1}^{\infty }\frac{\varphi _{0}(\xi
,\lambda _{n})\varphi _{0}(t,\lambda _{n})}{\alpha _{n}}d\xi dt+
\end{equation*}%
\begin{equation*}
+\int_{0}^{x}\frac{1-\sqrt{\rho \left( 2a-t\right) }}{1+\sqrt{\rho \left(
2a-t\right) }}A(2a-t)\int_{0}^{a}A(\xi )\sum_{n=1}^{\infty }\frac{\varphi
_{0}(\xi ,\lambda _{n})\varphi _{0}(t,\lambda _{n})}{\alpha _{n}}d\xi dt+
\end{equation*}%
\begin{equation*}
+\int_{0}^{x}\frac{2\rho (t)}{1+\sqrt{\rho \left( t\right) }}A(\mu
^{+}(t))\int_{a}^{\alpha x-\alpha a+a}A(\xi )\times
\end{equation*}%
\begin{equation*}
\times \sum_{n=1}^{\infty }\frac{2\alpha }{1+\alpha }\frac{\varphi
_{0}\left( \frac{\xi }{\alpha }+a-\frac{a}{\alpha },\lambda _{n}\right)
\varphi _{0}\left( t,\lambda _{n}\right) }{\alpha _{n}}d\xi dt+
\end{equation*}%
\begin{equation*}
+\int_{0}^{x}\frac{1-\sqrt{\rho \left( 2a-t\right) }}{1+\sqrt{\rho \left(
2a-t\right) }}A(2a-t)\int_{a}^{\alpha x-\alpha a+a}A(\xi )\times
\end{equation*}%
\begin{equation*}
\times \sum_{n=1}^{\infty }\frac{2\alpha }{1+\alpha }\frac{\varphi
_{0}\left( \frac{\xi }{\alpha }+a-\frac{a}{\alpha },\lambda _{n}\right)
\varphi _{0}\left( t,\lambda _{n}\right) }{\alpha _{n}}d\xi dt+
\end{equation*}%
\begin{equation*}
+\int_{0}^{x}\frac{2\rho (t)}{1+\sqrt{\rho \left( t\right) }}A(\mu
^{+}(t))\int_{a}^{\alpha x-\alpha a+a}A(\xi )\times
\end{equation*}%
\begin{equation*}
\times \sum_{n=1}^{\infty }\frac{1-\alpha }{1+\alpha }\frac{\varphi
_{0}\left( 2a-\xi ,\lambda _{n}\right) \varphi _{0}\left( t,\lambda
_{n}\right) }{\alpha _{n}}d\xi dt+
\end{equation*}%
\begin{equation*}
+\int_{0}^{x}\frac{1-\sqrt{\rho \left( 2a-t\right) }}{1+\sqrt{\rho \left(
2a-t\right) }}A(2a-t)\int_{a}^{\alpha x-\alpha a+a}A(\xi )\times
\end{equation*}%
\begin{equation*}
\times \sum_{n=1}^{\infty }\frac{1-\alpha }{1+\alpha }\frac{\varphi
_{0}\left( 2a-\xi ,\lambda _{n}\right) \varphi _{0}\left( t,\lambda
_{n}\right) }{\alpha _{n}}d\xi dt-
\end{equation*}%
\begin{equation*}
-\int_{0}^{x}\frac{2\rho (t)}{1+\sqrt{\rho \left( t\right) }}A(\mu
^{+}(t))\int_{0}^{a}A(\xi )\sum_{n=1}^{\infty }\frac{\varphi _{0}(\xi
,\lambda _{n}^{0})\varphi _{0}(t,\lambda _{n}^{0})}{\alpha _{n}^{0}}d\xi dt-
\end{equation*}%
\begin{equation*}
-\int_{0}^{x}\frac{1-\sqrt{\rho \left( 2a-t\right) }}{1+\sqrt{\rho \left(
2a-t\right) }}A(2a-t)\int_{0}^{a}A(\xi )\sum_{n=1}^{\infty }\frac{\varphi
_{0}(\xi ,\lambda _{n}^{0})\varphi _{0}(t,\lambda _{n}^{0})}{\alpha _{n}^{0}}%
d\xi dt-
\end{equation*}%
\begin{equation*}
-\int_{0}^{x}\frac{2\rho (t)}{1+\sqrt{\rho \left( t\right) }}A(\mu
^{+}(t))\int_{a}^{\alpha x-\alpha a+a}A(\xi )\times
\end{equation*}%
\begin{equation*}
\times \sum_{n=1}^{\infty }\frac{2\alpha }{1+\alpha }\frac{\varphi
_{0}\left( \frac{\xi }{\alpha }+a-\frac{a}{\alpha },\lambda _{n}^{0}\right)
\varphi _{0}\left( t,\lambda _{n}^{0}\right) }{\alpha _{n}^{0}}d\xi dt-
\end{equation*}%
\begin{equation*}
-\int_{0}^{x}\frac{1-\sqrt{\rho \left( 2a-t\right) }}{1+\sqrt{\rho \left(
2a-t\right) }}A(2a-t)\int_{a}^{\alpha x-\alpha a+a}A(\xi )\times
\end{equation*}%
\begin{equation*}
\times \sum_{n=1}^{\infty }\frac{2\alpha }{1+\alpha }\frac{\varphi
_{0}\left( \frac{\xi }{\alpha }+a-\frac{a}{\alpha },\lambda _{n}^{0}\right)
\varphi _{0}\left( t,\lambda _{n}^{0}\right) }{\alpha _{n}^{0}}d\xi dt-
\end{equation*}%
\begin{equation*}
-\int_{0}^{x}\frac{2\rho (t)}{1+\sqrt{\rho \left( t\right) }}A(\mu
^{+}(t))\int_{a}^{\alpha x-\alpha a+a}A(\xi )\times
\end{equation*}%
\begin{equation*}
\times \sum_{n=1}^{\infty }\frac{1-\alpha }{1+\alpha }\frac{\varphi
_{0}\left( 2a-\xi ,\lambda _{n}^{0}\right) \varphi _{0}\left( t,\lambda
_{n}^{0}\right) }{\alpha _{n}^{0}}d\xi dt-
\end{equation*}%
\begin{equation*}
-\int_{0}^{x}\frac{1-\sqrt{\rho \left( 2a-t\right) }}{1+\sqrt{\rho \left(
2a-t\right) }}A(2a-t)\int_{a}^{\alpha x-\alpha a+a}A(\xi )\times
\end{equation*}%
\begin{equation*}
\times \sum_{n=1}^{\infty }\frac{1-\alpha }{1+\alpha }\frac{\varphi
_{0}\left( 2a-\xi ,\lambda _{n}^{0}\right) \varphi _{0}\left( t,\lambda
_{n}^{0}\right) }{\alpha _{n}^{0}}d\xi dt=0.
\end{equation*}%
Substituting $\xi \rightarrow \frac{\xi }{\alpha }+a-\frac{a}{\alpha }$ in
third, fourth, ninth, and tenth double integrals and $\xi \rightarrow
2\alpha -\xi $ in fifth, sixth, eleventh and twelfth double integrals we get 
\begin{equation*}
\int_{0}^{x}\rho (t)\left( \frac{2}{1+\sqrt{\rho \left( t\right) }}A(\mu
^{+}(t))+\frac{1-\sqrt{\rho \left( 2a-t\right) }}{1+\sqrt{\rho \left(
2a-t\right) }}A(2a-t)\right) ^{2}dt+
\end{equation*}%
\begin{equation*}
+\int_{0}^{x}\frac{2\rho (t)}{1+\sqrt{\rho \left( t\right) }}A(\mu
^{+}(t))\int_{0}^{a}A(\xi )\sum_{n=1}^{\infty }\frac{\varphi _{0}(\xi
,\lambda _{n})\varphi _{0}(t,\lambda _{n})}{\alpha _{n}}d\xi dt+
\end{equation*}%
\begin{equation*}
+\int_{0}^{x}\frac{1-\sqrt{\rho \left( 2a-t\right) }}{1+\sqrt{\rho \left(
2a-t\right) }}A(2a-t)\int_{0}^{a}A(\xi )\sum_{n=1}^{\infty }\frac{\varphi
_{0}(\xi ,\lambda _{n})\varphi _{0}(t,\lambda _{n})}{\alpha _{n}}d\xi dt+
\end{equation*}%
\begin{equation*}
+\int_{0}^{x}\frac{2\rho (t)}{1+\sqrt{\rho \left( t\right) }}A(\mu
^{+}(t))\int_{a}^{x}A(\mu ^{+}(\xi ))\times
\end{equation*}%
\begin{equation*}
\times \sum_{n=1}^{\infty }\frac{2\alpha ^{2}}{1+\alpha }\frac{\varphi
_{0}(\xi ,\lambda _{n})\varphi _{0}(t,\lambda _{n})}{\alpha _{n}}d\xi dt+
\end{equation*}%
\begin{equation*}
+\int_{0}^{x}\frac{1-\sqrt{\rho \left( 2a-t\right) }}{1+\sqrt{\rho \left(
2a-t\right) }}A(2a-t)\int_{a}^{x}A(\mu ^{+}(\xi ))\times
\end{equation*}%
\begin{equation*}
\times \sum_{n=1}^{\infty }\frac{2\alpha ^{2}}{1+\alpha }\frac{\varphi
_{0}(\xi ,\lambda _{n})\varphi _{0}(t,\lambda _{n})}{\alpha _{n}}d\xi dt+
\end{equation*}%
\begin{equation*}
+\int_{0}^{x}\frac{2\rho (t)}{1+\sqrt{\rho \left( t\right) }}A(\mu
^{+}(t))\int_{-\alpha x+\alpha a+a}^{a}A(2a-\xi )\times
\end{equation*}%
\begin{equation*}
\times \sum_{n=1}^{\infty }\frac{1-\alpha }{1+\alpha }\frac{\varphi
_{0}\left( \xi ,\lambda _{n}\right) \varphi _{0}\left( t,\lambda _{n}\right) 
}{\alpha _{n}}d\xi dt+
\end{equation*}%
\begin{equation*}
+\int_{0}^{x}\frac{1-\sqrt{\rho \left( 2a-t\right) }}{1+\sqrt{\rho \left(
2a-t\right) }}A(2a-t)\int_{-\alpha x+\alpha a+a}^{a}A(2a-\xi )\times
\end{equation*}%
\begin{equation*}
\times \sum_{n=1}^{\infty }\frac{1-\alpha }{1+\alpha }\frac{\varphi
_{0}\left( \xi ,\lambda _{n}\right) \varphi _{0}\left( t,\lambda _{n}\right) 
}{\alpha _{n}}d\xi dt-
\end{equation*}%
\begin{equation*}
-\int_{0}^{x}\frac{2\rho (t)}{1+\sqrt{\rho \left( t\right) }}A(\mu
^{+}(t))\int_{0}^{a}A(\xi )\sum_{n=1}^{\infty }\frac{\varphi _{0}(\xi
,\lambda _{n}^{0})\varphi _{0}(t,\lambda _{n}^{0})}{\alpha _{n}^{0}}d\xi dt-
\end{equation*}%
\begin{equation*}
-\int_{0}^{x}\frac{1-\sqrt{\rho \left( 2a-t\right) }}{1+\sqrt{\rho \left(
2a-t\right) }}A(2a-t)\int_{0}^{a}A(\xi )\times
\end{equation*}%
\begin{equation*}
\times \sum_{n=1}^{\infty }\frac{\varphi _{0}(\xi ,\lambda _{n}^{0})\varphi
_{0}(t,\lambda _{n}^{0})}{\alpha _{n}^{0}}d\xi dt-
\end{equation*}%
\begin{equation*}
-\int_{0}^{x}\frac{2\rho (t)}{1+\sqrt{\rho \left( t\right) }}A(\mu
^{+}(t))\int_{a}^{x}A(\mu ^{+}(\xi ))\times
\end{equation*}%
\begin{equation*}
\times \sum_{n=1}^{\infty }\frac{2\alpha ^{2}}{1+\alpha }\frac{\varphi
_{0}(\xi ,\lambda _{n}^{0})\varphi _{0}(t,\lambda _{n}^{0})}{\alpha _{n}^{0}}%
d\xi dt-
\end{equation*}%
\begin{equation*}
-\int_{0}^{x}\frac{1-\sqrt{\rho \left( 2a-t\right) }}{1+\sqrt{\rho \left(
2a-t\right) }}A(2a-t)\int_{a}^{x}A(\mu ^{+}(\xi ))\times
\end{equation*}%
\begin{equation*}
\times \sum_{n=1}^{\infty }\frac{2\alpha ^{2}}{1+\alpha }\frac{\varphi
_{0}(\xi ,\lambda _{n}^{0})\varphi _{0}(t,\lambda _{n}^{0})}{\alpha _{n}^{0}}%
d\xi dt-
\end{equation*}%
\begin{equation*}
-\int_{0}^{x}\frac{2\rho (t)}{1+\sqrt{\rho \left( t\right) }}A(\mu
^{+}(t))\int_{-\alpha x+\alpha a+a}^{a}A(2a-\xi )\times
\end{equation*}%
\begin{equation*}
\times \sum_{n=1}^{\infty }\frac{1-\alpha }{1+\alpha }\frac{\varphi
_{0}\left( \xi ,\lambda _{n}^{0}\right) \varphi _{0}\left( t,\lambda
_{n}^{0}\right) }{\alpha _{n}^{0}}d\xi dt-
\end{equation*}%
\begin{equation*}
-\int_{0}^{x}\frac{1-\sqrt{\rho \left( 2a-t\right) }}{1+\sqrt{\rho \left(
2a-t\right) }}A(2a-t)\int_{-\alpha x+\alpha a+a}^{a}A(2a-\xi )\times
\end{equation*}%
\begin{equation*}
\times \sum_{n=1}^{\infty }\frac{1-\alpha }{1+\alpha }\frac{\varphi
_{0}\left( \xi ,\lambda _{n}^{0}\right) \varphi _{0}\left( t,\lambda
_{n}^{0}\right) }{\alpha _{n}^{0}}d\xi dt=0,
\end{equation*}%
from which we have 
\begin{equation*}
\int_{0}^{x}\rho (t)\left( \frac{2}{1+\sqrt{\rho \left( t\right) }}A(\mu
^{+}(t))+\frac{1-\sqrt{\rho \left( 2a-t\right) }}{1+\sqrt{\rho \left(
2a-t\right) }}A(2a-t)\right) ^{2}dt+
\end{equation*}%
\begin{equation*}
+\int_{0}^{x}\frac{2\rho (t)}{1+\sqrt{\rho \left( t\right) }}A(\mu
^{+}(t))\int_{0}^{x}\frac{2\rho (\xi )}{1+\sqrt{\rho \left( \xi \right) }}%
\times
\end{equation*}%
\begin{equation*}
\times A(\mu ^{+}(\xi ))\sum_{n=1}^{\infty }\frac{\varphi _{0}\left( \xi
,\lambda _{n}\right) \varphi _{0}\left( t,\lambda _{n}\right) }{\alpha _{n}}%
d\xi dt+
\end{equation*}%
\begin{equation*}
+\int_{0}^{x}\frac{1-\sqrt{\rho \left( 2a-t\right) }}{1+\sqrt{\rho \left(
2a-t\right) }}A(2a-t)\int_{0}^{x}\frac{2\rho (\xi )}{1+\sqrt{\rho \left( \xi
\right) }}\times
\end{equation*}%
\begin{equation*}
\times A(\mu ^{+}(\xi ))\sum_{n=1}^{\infty }\frac{\varphi _{0}\left( \xi
,\lambda _{n}\right) \varphi _{0}\left( t,\lambda _{n}\right) }{\alpha _{n}}%
d\xi dt+
\end{equation*}%
\begin{equation*}
+\int_{0}^{x}\frac{2\rho (t)}{1+\sqrt{\rho \left( t\right) }}A(\mu
^{+}(t))\int_{0}^{x}\frac{1-\sqrt{\rho \left( 2a-\xi \right) }}{1+\sqrt{\rho
\left( 2a-\xi \right) }}\times
\end{equation*}%
\begin{equation*}
\times A(2a-\xi )\sum_{n=1}^{\infty }\frac{\varphi _{0}\left( \xi ,\lambda
_{n}\right) \varphi _{0}\left( t,\lambda _{n}\right) }{\alpha _{n}}d\xi dt+
\end{equation*}%
\begin{equation*}
+\int_{0}^{x}\frac{1-\sqrt{\rho \left( 2a-t\right) }}{1+\sqrt{\rho \left(
2a-t\right) }}A(2a-t)\int_{0}^{x}\frac{1-\sqrt{\rho \left( 2a-\xi \right) }}{%
1+\sqrt{\rho \left( 2a-\xi \right) }}\times
\end{equation*}%
\begin{equation*}
\times A(2a-\xi )\sum_{n=1}^{\infty }\frac{\varphi _{0}\left( \xi ,\lambda
_{n}\right) \varphi _{0}\left( t,\lambda _{n}\right) }{\alpha _{n}}d\xi dt-
\end{equation*}%
\begin{equation*}
-\int_{0}^{x}\frac{2\rho (t)}{1+\sqrt{\rho \left( t\right) }}A(\mu
^{+}(t))\int_{0}^{x}\frac{2\rho (\xi )}{1+\sqrt{\rho \left( \xi \right) }}%
\times
\end{equation*}%
\begin{equation*}
\times A(\mu ^{+}(\xi ))\sum_{n=1}^{\infty }\frac{\varphi _{0}\left( \xi
,\lambda _{n}^{0}\right) \varphi _{0}\left( t,\lambda _{n}^{0}\right) }{%
\alpha _{n}^{0}}d\xi dt-
\end{equation*}%
\begin{equation*}
-\int_{0}^{x}\frac{1-\sqrt{\rho \left( 2a-t\right) }}{1+\sqrt{\rho \left(
2a-t\right) }}A(2a-t)\int_{0}^{x}\frac{2\rho (\xi )}{1+\sqrt{\rho \left( \xi
\right) }}\times
\end{equation*}%
\begin{equation*}
\times A(\mu ^{+}(\xi ))\sum_{n=1}^{\infty }\frac{\varphi _{0}\left( \xi
,\lambda _{n}^{0}\right) \varphi _{0}\left( t,\lambda _{n}^{0}\right) }{%
\alpha _{n}^{0}}d\xi dt-
\end{equation*}%
\begin{equation*}
-\int_{0}^{x}\frac{2\rho (t)}{1+\sqrt{\rho \left( t\right) }}A(\mu
^{+}(t))\int_{0}^{x}\frac{1-\sqrt{\rho \left( 2a-\xi \right) }}{1+\sqrt{\rho
\left( 2a-\xi \right) }}\times
\end{equation*}%
\begin{equation*}
\times A(2a-\xi )\sum_{n=1}^{\infty }\frac{\varphi _{0}\left( \xi ,\lambda
_{n}^{0}\right) \varphi _{0}\left( t,\lambda _{n}^{0}\right) }{\alpha
_{n}^{0}}d\xi dt-
\end{equation*}%
\begin{equation*}
-\int_{0}^{x}\frac{1-\sqrt{\rho \left( 2a-t\right) }}{1+\sqrt{\rho \left(
2a-t\right) }}A(2a-t)\int_{0}^{x}\frac{1-\sqrt{\rho \left( 2a-\xi \right) }}{%
1+\sqrt{\rho \left( 2a-\xi \right) }}\times
\end{equation*}%
\begin{equation*}
\times A(2a-\xi )\sum_{n=1}^{\infty }\frac{\varphi _{0}\left( \xi ,\lambda
_{n}^{0}\right) \varphi _{0}\left( t,\lambda _{n}^{0}\right) }{\alpha
_{n}^{0}}d\xi dt=0.
\end{equation*}%
Thus we obtain 
\begin{equation*}
\int_{0}^{x}\rho (t)\left( \frac{2}{1+\sqrt{\rho \left( t\right) }}A(\mu
^{+}(t))+\frac{1-\sqrt{\rho \left( 2a-t\right) }}{1+\sqrt{\rho \left(
2a-t\right) }}A(2a-t)\right) ^{2}dt+
\end{equation*}%
\begin{equation*}
+\sum_{n=1}^{\infty }\frac{1}{\alpha _{n}}\left( \int_{0}^{x}\rho (t)\left( 
\frac{2}{1+\sqrt{\rho \left( t\right) }}A(\mu ^{+}(t))+\frac{1-\sqrt{\rho
\left( 2a-t\right) }}{1+\sqrt{\rho \left( 2a-t\right) }}A(2a-t)\right)
\varphi _{0}\left( t,\lambda _{n}\right) dt\right) ^{2}-
\end{equation*}%
\begin{equation*}
-\sum_{n=1}^{\infty }\frac{1}{\alpha _{n}^{0}}\left( \int_{0}^{x}\rho
(t)\left( \frac{2}{1+\sqrt{\rho \left( t\right) }}A(\mu ^{+}(t))+\frac{1-%
\sqrt{\rho \left( 2a-t\right) }}{1+\sqrt{\rho \left( 2a-t\right) }}%
A(2a-t)\right) \varphi _{0}\left( t,\lambda _{n}^{0}\right) dt\right) ^{2}=0.
\end{equation*}%
Using the Parseval's equality 
\begin{equation*}
\int_{0}^{x}\rho (t)f^{2}(t)dt=\sum_{n=1}^{\infty }\frac{1}{\alpha _{n}^{0}}%
\left( \int_{0}^{x}\rho (t)f(t)\varphi _{0}\left( t,\lambda _{n}^{0}\right)
dt\right) ^{2}
\end{equation*}%
for the function 
\begin{equation*}
f(t)=\frac{2}{1+\sqrt{\rho \left( t\right) }}A(\mu ^{+}(t))+\frac{1-\sqrt{%
\rho \left( 2a-t\right) }}{1+\sqrt{\rho \left( 2a-t\right) }}A(2a-t)\in
L_{2}(0,x)
\end{equation*}%
we have 
\begin{equation*}
\sum_{n=1}^{\infty }\frac{1}{\alpha _{n}}\left( \int_{0}^{x}\rho (t)\left( 
\frac{2}{1+\sqrt{\rho \left( t\right) }}A(\mu ^{+}(t))+\frac{1-\sqrt{\rho
\left( 2a-t\right) }}{1+\sqrt{\rho \left( 2a-t\right) }}A(2a-t)\right)
\varphi _{0}\left( t,\lambda _{n}\right) dt\right) ^{2}=0
\end{equation*}%
or 
\begin{equation*}
\int_{0}^{x}\rho (t)\left( \frac{2}{1+\sqrt{\rho \left( t\right) }}A(\mu
^{+}(t))+\frac{1-\sqrt{\rho \left( 2a-t\right) }}{1+\sqrt{\rho \left(
2a-t\right) }}A(2a-t)\right) \varphi _{0}\left( t,\lambda _{n}\right) dt=0,%
\quad n\geq 1.
\end{equation*}%
Since the system $\left\{ \varphi _{0}\left( t,\lambda _{n}\right) \right\}
_{n \geq 1}$ is compete in $L_{2,\rho }\left( 0,\pi \right) ,$ we have 
\begin{equation*}
\frac{2}{1+\sqrt{\rho \left( t\right) }}A(\mu ^{+}(t))+\frac{1-\sqrt{\rho
\left( 2a-t\right) }}{1+\sqrt{\rho \left( 2a-t\right) }}A(2a-t)=0,
\end{equation*}%
i.e. $\left( L_{x}A\right) (t)=0,$ where the operator $L_{x}$ is defined by (%
\ref{21}). From invertibility of $L_{x}$ in $L_{2,\rho }\left( 0,\pi \right) 
$ we get $A(x,.)=0.$
\end{proof}

\begin{theorem}
Let $L$ and $\overset{\sim }{L}$ be two boundary value problems and%
\begin{equation*}
\lambda _{n}=\overset{\sim }{\lambda _{n}},\text{ }\alpha _{n}=\overset{\sim 
}{\alpha _{n}},\text{ }\left( n\in \mathbf{Z}\right) .
\end{equation*}%
Then 
\begin{equation*}
q(x)=\overset{\sim }{q}(x)\text{ }x\in \lbrack 0,\pi ].
\end{equation*}
\end{theorem}

\begin{proof}
According to (\ref{9}) and (\ref{10}) $F_{0}(x,t)=\overset{\sim }{F_{0}}%
(x,t) $ and $F(x,t)=\overset{\sim }{F}(x,t)$. Then from the main
equation (\ref{8}), we have $A(x,t)=\overset{\sim }{A}(x,t)$. It follows
from (\ref{5}) that $q(x)=\overset{\sim }{q}(x)$ $x\in \lbrack 0,\pi ]$.
\end{proof}

\section{Example}
Using \cite{Akh2}, we can transform the main equation (\ref{8}) to the following equation:  
\begin{equation}
\tilde{A}(x,t)+F(x,t)+\int_0^x\tilde{A}(x,\xi)F(x,\xi)d\xi, \quad 0<t<x, \label{23}
\end{equation}
where 
\begin{equation}
F(x,t)=\rho(t)\sum_{n=1}^\infty\left(\frac{\varphi_0(t,\lambda_n)\varphi_0(x,\lambda_n)}{\alpha_n}+\frac{\varphi_0(t,\lambda_n^0)\varphi_0(x,\lambda_n^0)}{\alpha_n^0}\right)\label{24}
\end{equation}
and
\begin{equation*}
\tilde{A}(x,t)=\left\{
\begin{array}{c}
A(x,t) \quad , \quad 0<t<x \\
A(x,t) \quad , \quad 0<t<-\alpha x+\alpha a+a \\
A(x,t)+\frac{1-\alpha}{1+\alpha}A(x,2a-t) \quad ,\quad -\alpha x-\alpha a+a<t<a<x \\
\frac{2\alpha^2}{1+\alpha} A(x,\alpha t-\alpha a+a) \quad , \quad a<t<x .
\end{array}
\right.
\end{equation*}
We assume that $\lambda_n=\lambda_n^0=\frac{\pi}{\mu^+(\pi)}\left(n-\frac{1}{2}\right), \quad 
n \geq 1;\quad \alpha_n=\pi, \quad n>1; \quad \alpha_1=\frac{\pi}{2}; \quad \alpha_n^0=\pi, \quad n\geq 1.$
From the formula (\ref{24}), we obtain
\begin{equation}
F(x,t)=\frac{1}{\pi}\rho(t)\varphi_0(t,\lambda_1)\varphi_0(x,\lambda_1).\label{25}
\end{equation}
Substituting (\ref{25}) into the main equation (\ref{23}) we obtain 
\begin{equation*}
\tilde{A}(x,t)=-\frac{1}{\pi}\rho(t)\varphi_0(t,\lambda_1)\left[\varphi_0(x,\lambda_1)+\int_0^x\tilde{A}(x,\xi)\varphi_0(\xi,\lambda_1)d\xi\right],
\end{equation*}
\begin{equation}
\tilde{A}(x,t)=-\frac{1}{\pi}\rho(t)\varphi_0(t,\lambda_1)L(x),\label{26}
\end{equation}
where 
\begin{equation}
L(x)=\varphi_0(x,\lambda_1)+\int_0^x\tilde{A}(x,\xi)\varphi_0(\xi,\lambda_1)d\xi.\label{27}
\end{equation}
Substituting (\ref{26}) into (\ref{27}) we obtain 
\begin{equation*}
L(x)=\varphi_0(x,\lambda_1)-\frac{1}{\pi}L(x)\int_0^x\rho(\xi)\varphi_0^2(\xi,\lambda_1)d\xi.
\end{equation*}
Using the (\ref{6}), we calculate the integral into last equation. Then, we have
\begin{equation}
L(x)=\left\{
\begin{array}{c}
\varphi_0(x,\lambda_1)-\frac{1}{2\pi}L(x)\left(x+\frac{sin2\lambda_1x}{2\lambda_1}\right) \quad , \quad 0<x<a \\
\varphi_0(x,\lambda_1)-\frac{1}{2\pi}L(x)\Phi(x,\lambda_1) \quad,\quad a<x<\pi,
\end{array}\right.
\label{28}
\end{equation} 
where
\begin{equation*}
\Phi(x,\lambda_1)=\frac{a}{2}+\frac{sin2\lambda_1a}{4\lambda_1}+
\end{equation*}
\begin{equation*}
+\frac{1}{8}(\alpha+1)^2\left[x-a+\frac{1}{2\lambda_1}\left(\frac{sin2\lambda_1(\alpha x-a\alpha+a)}{\alpha}-sin2\lambda_1a\right)\right]+
\end{equation*}
\begin{equation*}
+\frac{1}{2}(\alpha-1)^2\left[x-a-\frac{1}{2\lambda_1}\left(\frac{sin2\lambda_1(-\alpha x+a\alpha+a)}{\alpha}-sin2\lambda_1(2a-x)\right)\right].
\end{equation*}
From (\ref{28}),
\begin{equation}
L(x)=\left\{
\begin{array}{c}
\varphi_0(x,\lambda_1)\left[1+\frac{1}{2\pi}\left(x+\frac{sin2\lambda_1x}{2\lambda_1}\right)\right]^{-1} \quad , \quad  0<x<a \\
\varphi_0(x,\lambda_1)\left[1+\frac{1}{\pi}\Phi(x,\lambda_1)\right]^{-1} \quad , \quad a<x<\pi,
\end{array}
\right.
\label{29}
\end{equation} 
Substituting (\ref{29}) into (\ref{26}) we get
\begin{equation}
\tilde{A}(x,t)=\left\{
\begin{array}{c}
-\frac{1}{\pi}\varphi_0(t,\lambda_1)\varphi_0(x,\lambda_1)\left[1+\frac{1}{2\pi}\left(x+\frac{sin2\lambda_1x}{2\lambda_1}\right)\right]^{-1}  \quad , \quad 0<x<a \\
-\frac{1}{\pi}\rho(t)\varphi_0(t,\lambda_1)\varphi_0(x,\lambda_1)\left[1+\frac{1}{\pi}\Phi(x,lambda_1)\right]^{-1} \quad ,\quad a<x<\pi,
\end{array}\right.
\label{30}
\end{equation}
where $0<t<x$ and $\lambda_1=\frac{\pi}{2\mu^+(\pi)}$.
Thus, we obtain the solution of main equation (\ref{8}). If we use the formula (\ref{5}) then, we obtain the potential $q(x)$.

\section*{Acknowledgements}

This work is supported by the Scientific and Technological Research Council of Turkey (TUBITAK).

\end{document}